\documentclass[a4paper,10pt]{amsart}
\usepackage{amsmath,amscd,amssymb,amsthm,amsfonts}
\usepackage{enumerate}
\usepackage{bm}
\usepackage{tensor}
\usepackage[dvips]{graphicx}
\usepackage[numbers]{natbib}
\usepackage[all]{xy}

\newtheorem{theorem}{Theorem}[section]

\newtheorem{corollary}[theorem]{Corollary}

\theoremstyle{definition}
\newtheorem{definition}[theorem]{Definition}

\theoremstyle{remark}
\newtheorem{remark}[theorem]{Remark}

\numberwithin{equation}{section}

\def\cF{{\mathcal F}}

\def\Flin{\cF{^\ell}}

\def\rv{\mathbf{v}}
\def\rh{\mathbf{h}}

\def\to{\rightarrow}

\def\to{\rightarrow}

\def\r#1{|_{#1}}

\usepackage{xcolor}
\definecolor{red}{rgb}{0.8,0,0.1}

\definecolor{blue}{rgb}{0.1,0.1,0.8}

\definecolor{green}{rgb}{0.01,0.8,0.32}

\title[]{The Closure of a Linear Foliations}

\author[de Melo]{Mateus de Melo}
\author[Struchiner]{Ivan Struchiner}

\address{M. de Melo, \hfill\break\indent 
Universidade Federal do Espírito Santo, Departamento de Matem\'{a}tica, \hfill\break\indent 
Avenida Fernando Ferrari 514, 29075-910, Vitória, Brazil.}
\email{\textbf{(de Melo)  mateus.melo@ufes.br}  }

\address{I. Struchiner, \hfill\break\indent 
Universidade de S\~{a}o Paulo, Instituto de Matem\'{a}tica e Estat\'{i}stica, \hfill\break\indent 
Rua do Mat\~{a}o 1010, 05508-090 S\~{a}o Paulo, Brazil.}
\email{\textbf{ivanstru@ime.usp.br}}



\thanks{
The first author was supported by grant $\#$ 2019/14777-3, São Paulo Research Foundation (FAPESP).
The second author was supported by grant $\#$ 2022/16310-8, São Paulo Research Foundation (FAPESP).
}

\begin{document}

\begin{abstract}
This paper presents a simplified geometric proof of the Molino-Alexandrino-Radeschi (MAR) Theorem, which states that the closure of a singular Riemannian foliation on a complete Riemannian manifold is itself a smooth singular Riemannian foliation. Our approach circumvents several technical and analytical tools employed in the previous proof of the Theorem, resulting in a more direct geometric demonstration. We first establish conditions for a projectable foliation to be Riemannian, focusing on compatible connections. We then apply these results to linear foliations on vector bundles and their lifts to frame bundles. Finally, we use these findings to the linearization of singular Riemannian foliations around leaf closures. This method allows us to prove the smoothness of the closure directly for the linear semi-local model, bypassing the need for intermediate results on orbit-like foliations. 
\end{abstract}

\maketitle

\section{Introduction}

In this paper, we present a simplified geometric proof of the Molino-Alexandrino-Radeschi (MAR) Theorem. This theorem states that the foliation obtained by taking closures of the leaves of a singular Riemannian foliation on a complete Riemannian manifold is itself a smooth singular Riemannian foliation. Our approach circumvents several technical and analytical tools employed in the proof of the Molino conjecture given by Alexandrino and Radeschi \cite{ar2,ar1}, resulting in a more direct geometric demonstration.

Recall that a foliation $\mathcal{F}$ on a manifold $M$ is a \emph{smooth} decomposition of $M$ into a disjoint union of connected, immersed submanifolds called leaves. The smoothness condition of a foliation is expressed as the property that any vector tangent to a leaf at a point extends locally to a smooth vector field on $M$ which is still tangent to the leaves of the foliation at every point. A foliation is termed regular if all leaves have the same dimension, and singular otherwise. When one considers a Riemannian metric on the underlying manifold $M$, a particularly important class of foliations is that of Riemannian foliations. A foliation is called Riemannian if every geodesic that meets $\mathcal{F}$ orthogonally at a point remains orthogonal to $\mathcal{F}$ for all time.

The Molino conjecture has a rich history in the study of foliations. In \cite{molino77}, Molino demonstrated that the foliation obtained by taking the closures of the leaves of a transversely complete foliation is still a smooth foliation. By lifting a regular Riemannian foliation on a complete Riemannian manifold to the orthonormal frame bundle of the normal bundle of the foliation, one obtains a transversely complete foliation. Therefore, the smoothness of closures for Riemannian foliations immediately follows.

Molino's work in \cite{molino88} further revealed that singular Riemannian foliations admit stratifications based on leaf dimension. He observed that leaf closures remain within the same stratum and, noting that the restriction to each stratum constitutes a regular Riemannian foliation, concluded that each leaf closure forms a smooth manifold. While the orthogonality property of geodesics required for a Riemannian foliation (transnormal system) followed automatically, Molino conjectured that these closures form a smooth singular Riemannian foliation.

Molino's conjecture was proved in several particular cases, the most significant of which turned out to be the case of orbit-like foliations, demonstrated by Alexandrino and Radeschi in \cite{ar1}. Their approach employed analytic methods to establish the smoothness of isometric flows on the quotient space of isometric actions. Leveraging the fact that orbit-like foliations behave transversally as isometric actions, 
they applied their result to lift isometric flows on the leaf space of orbit-like foliations to smooth flows. This allowed them to conclude that the leaf closures of such foliations form a smooth foliation. 

The importance of the orbit-like foliation case was underscored by Alexandrino and Radeschi in \cite{ar2}, where they showed that the general case of Molino's conjecture for singular Riemannian foliations could be reduced to this specific scenario. Indeed, they demonstrated that to prove Molino's conjecture, it was sufficient to show it for either of two foliations constructed from the original one: the linearized foliation $\Flin$ on the normal bundle $\nu B$, where $B = \overline{L}$ is a leaf closure,  or the ``local closure'' foliation, an orbit-like foliation containing the linearized foliation with the same closure. These foliations were shown to be Riemannian with respect to a ``linear'' metric on $\nu B$, invariant under homotheties, constructed by Alexandrino and Radeschi. This metric was obtained through a series of steps involving the exponential map, linearization along fibers, and the construction of an affine connection \cite[Sec. 5]{ar2}.
Our approach is to use the results of the geometric paper \cite{ar2} to prove directly the smoothness of the closure of the linearized foliation, bypassing the orbit-like construction and the intermediary result in \cite{ar1}. This method not only simplifies the proof but also provides new geometric insights into the structure of singular Riemannian foliations.

More precisely, we consider the lift of the linearized foliation to the orthonormal frame bundle of the normal bundle of the leaf closure $B = \bar{L}$, obtaining a foliation $\widehat{\mathcal{F}}$ on $O(\nu B)$. This is the foliation that was obtained in \cite{aims} where it was shown to be regular. We are thus in the following situation: we have a regular foliation $\widehat{\mathcal{F}}$ which projects to a regular Riemannian foliation $\mathcal{F}_B$ on $B$. It is then natural to ask for the existence of a Riemannian metric on $O(\nu B)$ making $\widehat{\mathcal{F}}$ into a Riemannian foliation. This is the content of our main theorem which then implies the alternative proof of the MAR Theorem. 

We take the following steps:
\begin{itemize}
\item  In Theorem \ref{thm:riem-fol}, we consider a surjective submersion $\pi: P \to B$ and a regular foliation $\mathcal{F}$ on $P$ that is $\pi$-projectable to a regular Riemannian foliation $\mathcal{F}_B$ on $B$. We establish sufficient conditions for $\mathcal{F}$ to be a Riemannian foliation.
\item In Corollary \ref{cor:principal-bundle} we apply the previous result to the special case where $\pi: P \to B$ is a principal bundle.
\item In Section \ref{sec:linear-foliations} we discuss the process of lifting linear foliations on vector bundles to their frame bundles. We investigate connections compatible with the foliation and we prove, in Theorem \ref{thm:closure-lin-fol}, that under certain regularity conditions, the existence of such connections implies the smoothness of the foliation obtained by taking closures of the leaves of the linear foliation.
\item Finally, in Section \ref{sec:proof-MAR} we complete the alternative proof of the MAR Theorem by pointing to the relevant results in the literature which show that our results apply to this case. 
\end{itemize}

\subsection*{Acknowledgments}
We thank M. Alexandrino for fruitful discussions and comments. M. de Melo thanks the Differential Geometry Research Group of IME-USP for the support during visits to the institute.

\section{Riemannian Foliations on Bundles and Compatible Connections}

This section establishes the conditions under which a projectable foliation becomes Riemannian. We begin by introducing key definitions and then present our main theorem, which characterizes Riemannian foliations in terms of compatible connections.

\begin{definition}[Projectable Foliation]
Let $\pi: P \to B$ be a submersion. A foliation $\mathcal{F}$ on $P$ is called \emph{$\pi$-projectable} if its projection $\mathcal{F}_B$ defines a (smooth!) foliation on $B$.
\end{definition}

\begin{definition}[Compatible Ehresmann Connection]
Let $\mathcal{F}$ be a regular $\pi$-projectable foliation whose projection is also a regular foliation. An Ehresmann connection $\sigma: \pi^{*}TB \to TP$ is compatible with $\mathcal{F}$ if $\sigma(\pi^*T\mathcal{F}_B) \subset T\mathcal{F}$. This compatibility is represented by the following commutative diagram:
$$
  \xymatrix{
  0 \ar[r] & \ker d\pi \ar[r] & TP \ar[r]  & \pi^*TB \ar[r] \ar@/^/[l]^{\sigma} & 0 \\ 
  0 \ar[r] & T\mathcal{F}^{\rv} \ar[r]\ar[u] & T\mathcal{F} \ar[r]\ar[u]  & \pi^*T\mathcal{F}_B \ar[r]\ar[u] & 0.                                                                        
}       
$$
\end{definition}

Let $\mathcal{H}$ denote the image of $\sigma$ and $\mathcal{V}$ the kernel of $d\pi$. Then $TP = \mathcal{H} \oplus \mathcal{V}$. We may refer to $\mathcal{H}$ as a connection, as it uniquely determines $\sigma$ and vice versa.

\begin{remark}
It is easy to see that compatible Ehresmann connections always exist. For example, one can first choose a splitting $\sigma_{\mathcal{F}}: \pi^*T\mathcal{F}_B \to T\mathcal{F}$,  
then complementary subbundles $\mathcal{C}$ and $\mathcal{C}_B$ for $T\mathcal{F}$ and $T\mathcal{F}_B$, respectively, and lastly a splitting $\sigma_{\mathcal{C}}: \pi^*\mathcal{C}_B \to \mathcal{C}$.
\end{remark}

\begin{remark}[Induced Metric]
Given an Ehresmann connection $\mathcal{H}$ for $\pi: M \to B$, a metric $\eta_B$ on $B$, and a metric $\eta^{\rv}$ on $\mathcal{V}$, there is a unique metric $\eta$ on $M$ by lifting $\eta_B$ to $\mathcal{H}$ and declaring $\mathcal{V}$ orthogonal to $\mathcal{H}$.
Furthermore, $\eta$ is complete if and only if $\mathcal{H}$, $\eta_B$ and the restrictions of $\eta^{\rv}$ to the fibers are complete.
\end{remark}

{
Under the assumption that both $\mathcal{F}$ and $\mathcal{F}_B$ are regular foliations, a compatible connection and an induced metric decompose $TP$ as:

\[
TP = T\mathcal{F} \oplus T\mathcal{F}^{\perp} = \mathcal{T}^{\rh} \oplus \mathcal{T}^{\rv} \oplus \mathcal{N}^{\rh} \oplus \mathcal{N}^{\rv},
\]
where $\mathcal{T}^{\rh} = T\mathcal{F} \cap \mathcal{H}$, 
$\mathcal{T}^{\rv} = T\mathcal{F} \cap \mathcal{V}$, 
$\mathcal{N}^{\rh} = T\mathcal{F}^{\perp} \cap \mathcal{H}$, 
and $\mathcal{N}^{\rv} = T\mathcal{F}^{\perp} \cap \mathcal{V}$.
Observe that $\mathcal{T}^{\rv}$ is the tangent distribution of the foliation obtained by restricting $\mathcal{F}$ to the $\pi$-fibers.

\begin{definition}[$\mathcal{F}$-foliated Connection]
A compatible connection $\mathcal{H}$ is called an $\mathcal{F}$-foliated connection if for every point $x \in M$ and every vector $w \in \mathcal{N}_x^{\rh}$, there exists an extension of $w$ to an $\mathcal{F}$-foliated vector field, i.e., a vector field $W$ such that $[W, \Gamma(\cF)] \in \Gamma(\cF)$.
\end{definition}

We now state our main theorem characterizing Riemannian foliations:

\begin{theorem}\label{thm:riem-fol}
Let $\pi:P \to B$ be a submersion with $\mathcal{F}$ a regular $\pi$-projectable foliation whose projection is a regular foliation $\mathcal{F}_B$ in $B$.
Suppose $\mathcal{H}$ is a connection compatible with $\mathcal{F}$, $\eta_B$ is a metric on $B$, and $\eta^{\rv}$ is a metric on $\mathcal{V}$, all these inducing a metric $\eta$ on $P$.
Then $\mathcal{F}$ is a Riemannian foliation with respect to $\eta$ if and only if:
\begin{enumerate}[i)]
    \item $(B,\mathcal{F}_B,\eta_B)$ is a Riemannian foliation on the base;
    \item $\mathcal{H}$ is an $\mathcal{F}$-foliated connection;
    \item $\eta^{\rv}$ is preserved by the partial connection $\mathcal{H}\r{\mathcal{F}_B}$;
    \item $(\pi^{-1}(b),\mathcal{F}\r{\pi^{-1}(b)}, \eta^{\rv}\r{\pi^{-1}(b)})$ is a Riemannian foliation for every $b$ in $B$.
\end{enumerate}
Moreover, if $\mathcal{H}$, $\eta_B$ and the restrictions of $\eta^{\rv}$ to the fibers are complete, then the leaf closure foliation $\overline{\mathcal{F}}$ is smooth.
\end{theorem}

\begin{proof}
Recall that a regular foliation $\mathcal{F}$ is Riemannian with respect to $\eta$ if the restriction $\eta^{\perp}$ to $T\mathcal{F}^{\perp}$ is invariant along $\mathcal{F}$, i.e., $L_U\eta^{\perp}=0$ for all $U$ in $\mathfrak{X}(\mathcal{F})$.

Therefore, we will prove the theorem by examining the Lie derivative of $\eta^{\perp}$ along vector fields tangent to $\mathcal{F}$. Due to the $C^{\infty}(M)$-linearity of the Lie derivative of  $\eta^{\perp}$ along $\mathcal{F}$, it suffices to consider $\pi$-projectable vector fields $U$ tangent to $\mathcal{F}$.

Let $X$ and $Y$ be $\pi$-projectable vector fields in $T\mathcal{F}^{\perp}$. We consider three cases:

\noindent \textbf{1) $X, Y$ are tangent to $\mathcal{N}^{\rv}$:}

\ \
 \begin{enumerate}[a)]
 \item  If $U$ is tangent to $\mathcal{T}^{\rh}$:
      \begin{equation*}
         L_U\eta^{\perp}(X,Y) = L_U\eta(X,Y) = L_U\eta^{\rv}(X,Y).
      \end{equation*}
      This implies that $\eta^{\rv}$ must be preserved by the $\mathcal{F}_B$-partial connection induced by $\mathcal{H}$, which is condition (iii).
\\

 \item If $U$ is tangent to $\mathcal{T}^{\rv}$, we can restrict all involved vector fields to the fiber $\pi^{-1}(b)$, where $\pi(x)=b$. Then, $L_U\eta^{\perp}(X,Y)=0$ for all $U$ in $\mathcal{T}^{\rv}$ and $X,Y$ in $\mathcal{N}^{\rv}$ if and only if $(\pi^{-1}(b),\mathcal{F}\r{\pi^{-1}(b)}, \eta^{\rv}\r{\pi^{-1}(b)})$ is a Riemannian foliation for all $b\in B$, which is condition (iv).
 \end{enumerate}

\noindent \textbf{2) $X, Y$ are vector fields in $\mathcal{N}^{\rh}$:}

   \begin{align*}
      L_U\eta^{\perp}(X,Y) &= L_U\left(\eta(X,Y)\right)-\eta([U,X],Y) - \eta([U,Y],X) \\
      &= L_U \left(\pi^*\eta_B(X,Y)\right) - \pi^*\eta_B([U,X],Y) - \pi^*\eta_B([U,Y],X) \\
      &= L_{\pi_*U}\left(\eta_B(\pi_*X,\pi_*Y)\right) - \eta_B([\pi_*U,\pi_*X],\pi_*Y) - \eta_B([\pi_*U,\pi_*Y],\pi_*X) \\
      &= L_{\pi_*U}\eta_B^{\perp}(\pi_*X,\pi_*Y)
   \end{align*}
   This vanishes for all $U\in T\mathcal{F}$ and $X,Y$ in $\mathcal{N}^{\rh}$ if and only if $(B,\mathcal{F}_B,\eta_B)$ is a Riemannian foliation, which is condition (i).
\\

\noindent \textbf{3) $X$ is in $\mathcal{N}^{\rv}$ and $Y$ is in $\mathcal{N}^{\rh}$:}
   \begin{align*}
      \left(L_U\eta^{\perp}\right)(X,Y) &= L_U\left(\eta(X^\perp,Y^\perp)\right) -\eta([U,X]^\perp,Y^\perp) - \eta([U,Y],X) \\
      &= -\eta([U,Y],X)
   \end{align*}
   This vanishes pointwise if for all $w$ in $\mathcal{N}_x^{\rh}$ there is an extension to an $\mathcal{F}$-foliated vector field. This is precisely the definition of an $\mathcal{F}$-foliated connection, which is condition (ii). Conversely, if $\mathcal{F}$ is Riemannian, then $\mathcal{F}_B$ is also Riemannian. Hence, we can locally extend the projection of $w$ in $\mathcal{N}_x^{\rh}$ to a vector field $W'$ such that $[W',U'] = 0$ for all vector fields $U'$ tangent to $\mathcal{F}_B$. The vanishing of the above expression then implies that the lift $W$ to $M$ is foliated.

Thus, $\mathcal{F}$ is a Riemannian foliation with respect to $\eta$ if and only if all four conditions are satisfied.
\end{proof}

For principal bundles with invariant foliations, we have the following simplified result:

\begin{corollary}\label{cor:principal-bundle}
Let $\pi:P \to B$ be a $G$-principal bundle with $\mathcal{F}$ a $G$-invariant regular foliation whose projection is a regular foliation $\mathcal{F}_B$ on $B$.
Suppose $\mathcal{H}$ is a $G$-connection compatible with $\mathcal{F}$, $\eta_B$ is a metric on $B$, and $\langle\,,\rangle_G$ is a left-invariant metric on $G$, all these inducing a metric $\eta$ on $P$.
Then $\mathcal{F}$ is a Riemannian foliation with respect to $\eta$ if and only if:
\begin{enumerate}[i)]
    \item $(B,\mathcal{F}_B,\eta_B)$ is a Riemannian foliation on the base;
    \item $\mathcal{H}$ is an $\mathcal{F}$-foliated connection.
\end{enumerate}
Moreover, if $\eta_B$ is complete then the leafwise closure foliation  $\overline{\mathcal{F}}$ is a $G$-invariant smooth foliation. 
\end{corollary}

\begin{proof}
Condition (iii) of Theorem \ref{thm:riem-fol} is automatically satisfied because the holonomy of a principal connection acts as left translations in the fibers. Condition (iv) is satisfied because the restriction of a $G$-invariant foliation to a fiber is identified with the foliation of $G$ by left cosets of a connected subgroup, which is Riemannian with respect to any left-invariant metric. 
Since left-invariant metrics and principal connections are always complete, the completeness of $\eta$ is characterized by the completeness of $\eta_B$. Hence, if $\eta_B$ is complete, then $\mathcal{F}$ is a regular Riemannian foliation on a complete Riemannian manifold, and by Molino's theory, $\overline{\mathcal{F}}$ is a smooth foliation.
\end{proof}

\section{Linear Foliations} \label{sec:linear-foliations}

This section explores the relationship between linear foliations on vector bundles and their lifts to frame bundles. We introduce key definitions and present a theorem characterizing conditions under which the closure of a linear foliation is smooth.

Recall that a vector field $X$ on a vector bundle $E \to B$ is called linear if it is invariant under scalar multiplication (homotheties) of $E$. Equivalently, $X$ is projectable to $B$, and its flow, whenever defined, is a linear isomorphism between the fibers of $E$. Recall also that an affine connection on $E$ is an Ehresmann connection $\mathcal{H} \subset TE$ for the submersion $E \to B$ that is generated by linear vector fields, or equivalently, such that $\mathcal{H}$ is invariant under the differential of the scalar multiplication on $E$.

Let us denote by $\pi: \mathrm{Fr}(E) \to B$ the frame bundle of $E$. It is a principal $\mathrm{GL}(q)$-bundle whose fibers consist of all linear isomorphisms $\mathbb{R}^q \to E_b = \pi^{-1}(b)$, where $q$ is the rank of $E$. There exists a Lie algebra isomorphism between linear vector fields on $E$ and $\mathrm{GL}(q)$-invariant vector fields on $\mathrm{Fr}(E)$:
$$
\widehat{\,} : \mathfrak{X}(E)^{\ell} \to \mathfrak{X}(\mathrm{Fr}(E))^{\mathrm{GL}(q)}.
$$
It is defined by lifting the flows of the linear vector fields to $\mathrm{Fr}(E)$ and then taking the induced vector field. Explicitly,
\[\widehat{X}(p) = \frac{d}{dt}\vert_{t=0} \mathrm{Fl}^t_X \circ p.\]

The lifting of linear vector fields allows us to extend linear structures from a vector bundle to its frame bundle. This process transforms linear structures into invariant ones. For instance, an affine connection $\mathcal{H}$ on $E$ lifts to a principal connection $\widehat{\mathcal{H}}$ on the frame bundle of $E$. 

Our goal is to apply this lifting procedure to linear foliations on $E$, creating invariant foliations on $\mathrm{Fr}(E)$. However, we must note an important distinction. For regular foliations, the module of tangent vector fields and the leaf decomposition uniquely determine each other. This is not the case for singular foliations, where a module of tangent vector fields is in general not uniquely defined by the underlying foliation.

To address this, we adopt the following convention: in this section, a smooth foliation is defined as an involutive and locally finite module of vector fields. In particular, a \textbf{linear foliation} on $E$ will be identified with an involutive and locally finite module of \emph{linear} vector fields on $E$. By lifting the generating linear vector fields, we can lift a linear foliation $\cF$ on $E$ to an invariant foliation $\widehat{\cF}$ on $\mathrm{Fr}(E)$.

\begin{definition}[Metric-Preserving Linear Foliation]
Let $E$ be equipped with a fiberwise metric $\langle\,,\rangle$. A linear foliation $\mathcal{F}$ preserves $\langle\,,\rangle$ if the flow of the linear vector fields generating $\mathcal{F}$ are linear isometries between the fibers of $(E, \langle\,,\rangle)$.
\end{definition}

We now state our main theorem characterizing conditions for the smoothness of the closure of a linear foliation:

\begin{theorem}\label{thm:closure-lin-fol}    
Let $(E, \langle\,,\rangle, \mathcal{F}) \rightarrow B$ be a vector bundle equipped with a fiberwise metric and a linear foliation, where $\mathcal{F}$ preserves the metric. 
Suppose that both the lifted foliation $\widehat{\mathcal{F}}$ on the frame bundle $\mathrm{Fr}(E)$ and the induced foliation $\mathcal{F}_B$ on the base $B$ are regular.  
If $\mathcal{F}_B$ is a Riemannian foliation for some metric, and there exists a compatible $\mathcal{F}$-foliated affine connection $\mathcal{H}$, then the closure $\overline{\mathcal{F}}$ is a smooth foliation.
\end{theorem}

\begin{proof}
Since we have an isomorphism of Lie algebras between linear and invariant vector fields, the affine connection $\mathcal{H}$ lifts to a compatible $\widehat{\mathcal{F}}$-foliated principal connection on $\mathrm{Fr}(E)$. By Corollary \ref{cor:principal-bundle}, the foliation $\widehat{\mathcal{F}}$ is Riemannian, and its closure is a smooth $GL(q)$-invariant foliation. Thus, $\overline{\widehat{\mathcal{F}}}$ is smooth.

Let $\mathrm{O}(E) \subset \mathrm{Fr}(E)$ be the reduction to the orthonormal frames of the fiberwise metric on $E$. 
By hypothesis, $\mathrm{O}(E)$ is $\widehat{\mathcal{F}}$-saturated, and since $\mathrm{O}(E)$ is a closed subbundle, it is also $\overline{\widehat{\mathcal{F}}}$-saturated. Hence, the restriction of $\overline{\widehat{\mathcal{F}}}$ to $\mathrm{O}(E)$ is a smooth foliation.

Let $\pi:\mathrm{O}(E)\times \mathbb{R}^q \to E$ be the projection recovering $E$ as an associated bundle.
Observe that, because the structural group of $\mathrm{O}(E)$ is $O(q)$, this map is proper. 
Consider on $\mathrm{O}(E)\times \mathbb{R}^q$ the product foliation $\overline{\widehat{\mathcal{F}}} \times \{\text{pt}\}$, which projects to the smooth foliation $\mathcal{F}$ on $E$.
Moreover, since $\pi$ is proper and $\pi_{*}(\widehat{\mathcal{F}} \times \{\text{pt}\}) = \mathcal{F}$, we have $\pi_{*}(\overline{\widehat{\mathcal{F}}}\times\{\text{pt}\}) = \overline{\mathcal{F}}$, which shows that $\overline{\mathcal{F}}$ is smooth. 
\end{proof}

\section{Closure of Linear Semi-local Models for SRFs}\label{sec:proof-MAR}

This section applies the results from previous sections to prove the smoothness of the closure of linearized singular Riemannian foliations. We then use this to provide an alternative proof of the Molino-Alexandrino-Radeschi Theorem.

\begin{definition}[Linearized Foliation]
Let $\mathcal{F}$ be a singular Riemannian foliation on a complete Riemannian manifold $M$, and $B = \overline{L}$ be the closure of a leaf. The linearization of $\mathcal{F}$ around $B$, denoted $\mathcal{F}^{\ell}$, is a linear foliation on the normal bundle $\nu B$. As a module of vector fields, $\mathcal{F}^{\ell}$ is the maximal locally finite module of linear vector fields obtained by linearizing the vector fields tangent to $\mathcal{F}$ along $B$.
\end{definition}

We now state our main theorem:

\begin{theorem}\label{thm:smooth-closure-linear}
Let $\mathcal{F}$ be a Riemannian foliation on a complete Riemannian manifold. 
Suppose $\mathcal{F}^{\ell}$ is the linearization of $\mathcal{F}$ around the closure of a leaf. Then its closure, $\overline{\mathcal{F}^{\ell}}$, is a smooth foliation.
\end{theorem}

To prove this theorem, we need to establish several key properties of $\mathcal{F}^{\ell}$:

1. The projection of $\mathcal{F}^{\ell}$ along $\nu B \to B$ coincides with the regular foliation $\mathcal{F}_B$ obtained by restricting $\mathcal{F}$ to $B$ - details on the linearized foliation can be found in \cite[Sec. 2.5]{ar2}.

2. $\mathcal{F}^{\ell}$ lifts canonically to an invariant foliation $\widehat{\mathcal{F}}$ on the frame bundle of $\nu B$, see Section \ref{sec:linear-foliations}.

3. $\mathcal{F}^{\ell}$ preserves the fiberwise metric induced on $\nu B$, as shown in \cite[Prop. 14]{mr}.

4. $\widehat{\mathcal{F}}$ is a regular foliation on the orthonormal frame bundle and the full frame bundle, as demonstrated in \cite[Thm. 4.1]{aims}.

5. There exists a compatible $\mathcal{F}^{\ell}$-foliated affine connection $\mathcal{H}$ on $\nu B$, constructed in \cite[Sec. 5]{ar2}.

\begin{proof}
Given the properties listed above, we can directly apply Theorem \ref{thm:closure-lin-fol} to the linear semi-local model foliation $\mathcal{F}^{\ell}$. This immediately yields the result that $\overline{\mathcal{F}^{\ell}}$ is a smooth foliation.
\end{proof}

We can now use this result to provide an alternative proof of the Molino-Alexandrino-Radeschi Theorem:

\begin{theorem}[Molino-Alexandrino-Radeschi]
Let $\mathcal{F}$ be a singular Riemannian foliation on a complete Riemannian manifold. Then the closure $\overline{\mathcal{F}}$ is a smooth foliation.
\end{theorem}

\begin{proof}
Let $v$ be a vector tangent to the leaf closure $\overline{L} = B$. By Theorem \ref{thm:smooth-closure-linear}, $v$ has an extension to a smooth vector field $X$ tangent to $\overline{\mathcal{F}^{\ell}}$. 

The linearization process implies that $\mathcal{F}^{\ell} \subset \mathcal{F}$, which leads to $\overline{\mathcal{F}^{\ell}} \subset \overline{\mathcal{F}}$. Hence, $X$ is tangent to $\overline{\mathcal{F}}$.

This shows that any vector tangent to a leaf closure of $\mathcal{F}$ can be extended to a smooth vector field tangent to $\overline{\mathcal{F}}$, proving that $\overline{\mathcal{F}}$ is a smooth foliation.
\end{proof}


\bibliographystyle{alpha}
\bibliography{references.bib}

\end{document}